\documentclass[]{birkjour}
\usepackage{amsmath,amsthm,amsopn,amssymb,mathrsfs,times,newtxtext,newtxmath,upref}
\usepackage[mathscr]{eucal}

\usepackage{enumitem}
\setlist{label={$$\roman{enumi}\kern1pt$)$}}

\newtheorem{thm}{Theorem}[section]
\newtheorem{prop}[thm]{Proposition}
\newtheorem{cor}[thm]{Corollary}
\newtheorem*{cor*}{Corollary}
\newtheorem{lema}[thm]{Lemma}
\newtheorem*{lema*}{Lemma}

\numberwithin{equation}{section}
\theoremstyle{definition}
\newtheorem*{Def}{Definition}

\newtheorem{obs}{Remark}

\newcommand{\St}{\mathcal{S}}

\newcommand{\cQ}{{\bf Q}}

\newcommand{\M}{\mathcal{M}}

\newcommand{\N}{\mathcal{N}}

\newcommand{\X}{\mathcal{X}}
\newcommand{\Y}{\mathcal{Y}}
\newcommand{\Z}{\mathcal{Z}}
\newcommand{\W}{\mathcal{W}}

\newcommand{\mc}[1]{\mathcal{#1}}
\newcommand{\T}{\mathcal{T}}

\DeclareMathOperator{\Mp}{{\bf Mp}}
\DeclareMathOperator{\Id}{{\bf Id}}

\newcommand{\PMN}{P_{\M,\N}}

\DeclareMathOperator{\ran}{ran}
\DeclareMathOperator{\dom}{dom}

\DeclareMathOperator{\mul}{mul}

\newcommand{\pl}{{\mathbin{/\mkern-3mu/}}}

\newcommand{\lr}{{\bf L}}
\newcommand{\ls}{{\bf LS}}
\newcommand{\op}{{\bf O}}

\begin{document}


\title{Factorizations of linear relations by idempotents}

\author[Arias]{M.~Laura~Arias}

\address{
	Instituto Argentino de Matem\'atica ``Alberto P. Calder\'on'' \\
	CONICET\\
	Saavedra 15, Piso 3\\
	(1083) Buenos Aires, Argentina
}
\email{lauraarias@conicet.gov.ar}

\author[Contino]{Maximiliano Contino}

\address{Depto. de Matemáticas, CUNEF Universidad \\  C. de Leonardo Prieto Castro, 2 \\ (28040) Madrid, Spain \\ [5pt]
	Facultad de Ingenier\'{\i}a, Universidad de Buenos Aires\\
	Paseo Col\'on 850 \\
	(1063) Buenos Aires, Argentina}
\email{mcontino@fi.uba.ar}

\author[Marcantognini]{Stefania Marcantognini}

\address{
	Instituto Argentino de Matem\'atica ``Alberto P. Calder\'on'' \\
	CONICET\\
	Saavedra 15, Piso 3\\
	(1083) Buenos Aires, Argentina \\[5pt]
	Universidad Nacional de General Sarmiento -- Instituto de Ciencias
	\\ Juan Mar\'ia Gutierrez 
	\\ (1613) Los Polvorines, Pcia. de
	Buenos Aires, Argentina}
\email{smarcantognini@ungs.edu.ar}

\keywords{linear relations, idempotents, multivalued projections, factorizations}
\subjclass{47A06, 47A62, 47A68}

%

 \begin{abstract} 
We study the class of those linear relations that can be factorized as products of idempotent relations. We provide several characterizations of this class, extending known factorization results for operators to the more general setting of linear relations.
\end{abstract}

\maketitle

\section{Introduction}

The notion of a linear relation as a multivalued linear operator was introduced by R. Arens in \cite{Arens}, and has since been widely studied due to its broad range of applications. 

In this article, we focus our attention on factorizations of linear relations. A seminal work on the subject is \cite{PS}, where the authors characterize those pairs of linear relations $R$ and $S$ for which there exists a linear relation (or operator) $T$ such
that $R \subseteq ST,$ respectively $R \subseteq TS.$ The equalities $R=ST$ and $R=TS$ were studied later in \cite{Sando}. These results extend the well-known  Douglas' factorization theorem for  linear operators in Hilbert spaces \cite{Douglas} to the setting of linear relations.  Motivated by these articles, we investigate factorizations of the form 
\begin{equation}\label{eqintro}
R=SQ \ \text{ and} \ R=QS,
\end{equation}
 where $Q$ is an idempotent linear relation. 
 In operator theory, factorizations of the form (\ref{eqintro}) with $Q$ a projection (i.e., a single-valued idempotent relation) correspond to restrictions or extensions of operators. This perspective inspired us to explore similar factorizations in the context of linear relations.
 
 We begin by examining (\ref{eqintro}) with $R$ and $S$ linear relations and $Q$ a projection. To this end, the concept of linear selection is a useful tool for extending operator results to linear relations.  A linear selection of a linear relation  $R$ is a single-valued linear relation that is contained in $R$ and maintains the same domain as  $R$ \cite{LN}. 
 Using linear selections, we provide alternative proofs of Douglas-type results for linear relations, and characterize those $R$ and $S$ admitting factorizations of the form (\ref{eqintro}) with $Q$ a projection.
 
We then study factorizations (\ref{eqintro}) where $Q$ is a multivalued projection—an idempotent linear relation with invariant domain.
In many senses, multivalued projections serve as the linear relation analogue of projections.  This concept was introduced by Cross et al. in \cite{Cross}, and later studied in \cite{Labrousse}, \cite{ACMM2}, \cite{ACMM3}.  For a comprehensive survey on idempotent linear relations, see \cite{ACMM1}. 
We analyze  the domain, range, kernel and multivalued part of linear relations factored as in (\ref{eqintro}), where $Q$ is either an idempotent or a multivalued projection. Additionally, we investigate factorizations of the form $PQ$ with $P,Q$ multivalued projections, identifying various conditions under which such factorizations exist. 
These results extend known properties of products of orthogonal projections (or idempotent operators) in Hilbert spaces \cite{ACM}, \cite{Cor}, \cite{B}.

The paper is organized as follows. Section 2 introduces notation and background on linear relations, idempotents, and multivalued projections. Section 3 presents properties of idempotent relations that will be useful throughout the paper. Section 4 provides the definition of a linear selection and a result about the linear selections of idempotents and multivalued projections.
Section 5 contains the main results on factorizations (\ref{eqintro}) with $R, S$ linear relations and $Q$ a projection, see Theorems \ref{solQrl} and \ref{SolQ2}.
Finally, Section 6 addresses products of multivalued projections, with a complete description of this class given in Theorem \ref{Zorn} and its corollaries.

\section{Preliminaries}

Throughout the paper the symbols $\X, \Y$ and $\Z$ denote linear spaces over the real or complex field $\mathbb{K}$, and $\N, \St,\T$ linear subspaces of either of them.  A linear relation (multivalued operator) between $\X$ and $\Y$ is a linear subspace $T$ of the cartesian product $\X \times \Y$. The set of linear relations from $\X$ into $\Y$ is denoted by $\lr(\X,\Y),$ and $\lr(\X)$ when $\X=\Y.$ 

 Given $T\in \lr(\X,\Y),$ $\dom T,$  $\ran T$ and $\ker T$ stand for the domain,  range and kernel of $T,$ respectively. The multivalued part of $T$ is defined by $\mul T:=\{y\in \Y: (0,y)\in T\}.$ If $\mul T=\{0\},$ $T$ is (the graph of) an operator.  The set of linear operators from $\X$ into $\Y$ is denoted $\op(\X,\Y),$ and $\op(\X)$ when $\X=\Y.$ The inverse of $T$ is the relation $T^{-1}:=\{(y,x) : (x,y) \in T\}$. Thus,  $\dom T^{-1}=\ran T$ and $\mul T^{-1} =\ker T.$ 
 We denote
 $$T(\N) := \{y : (x, y) \in T \mbox{ for some } x \in \N \},$$
 and for any $x \in \dom T,$  $T(x) :=T(\{x\}).$

For $T, S \in \lr(\X,\Y),$ the sum of $T$ and $S$ as subspaces is written  $T \ \hat{+}  \ S$   and  $T\hat{\dotplus}S$ if $T\cap S=\{(0,0)\}$.

The sum $T+S$ is the linear relation defined by
$$T+S:=\{(x,y+z): (x,y ) \in T \mbox{ and } (x,z) \in S\}.$$
If $R \in \lr(\Y,\Z),$  the product $RT$ is the linear relation from $\X$ to $\Z $ defined by
$$RT:=\{(x,z): (x,y) \in T \mbox{ and } (y,z) \in R \mbox{ for some } y \in \Y\}.$$  
Given two subsets $A\subseteq \lr(\X,\Y) $ and $ B\subseteq \lr(\Y,\Z) $, we denote $BA:=\{RT: T\in A, R\in B\}.$

The following results will be used throughout the paper without mention.

\begin{lema}[{\cite[Lemma 2.3]{ACMM3}}] \label{sumasubs} Let $T, S \in \lr(\X, \Y),$ $R \in \lr(\Y,\Z)$ and $F\in \lr(\Z, \X)$. Then
	\begin{enumerate}
		\item[1.] $RT \ \hat{+} \ RS \subseteq R(T \ \hat{+} \ S) $ and the equality holds if $\ran T\subseteq \dom R$ or $\ran S\subseteq \dom R$.
		\item[2.] $TF \ \hat{+} \ SF \subseteq (T \ \hat{+} \ S)F $ and the equality holds if  $\dom T \subseteq \ran F$ or $\dom S\subseteq \ran F$.
	\end{enumerate}
\end{lema}

\begin{lema}(\cite[2.02]{Arens}, \cite[Proposition 1.21]{Labrousse})\label{lemalr} Let $S, T \in \lr(\X,\Y).$ Then $S=T$ if and only if $S\subseteq T,$ $\dom T\subseteq\dom S$ and $\mul T \subseteq \mul S.$
\end{lema}

\begin{lema}  The following statements hold:
\begin{enumerate}
	\item [1.] If $\St \dotplus \T=\N \dotplus \T$ and $\St \subseteq \N$ then $\St=\N.$
	\item[2.] Let $T \in \lr(\X,\Y)$. Then	$(\{0\}\times \N)T=\ker T \times \N \mbox{ and }T(\{0\}\times \St)=\{0\} \times T(\St).$ Also, if $\N \subseteq \ker T$ then $T \hat{+} (\N \times  \St )= T\hat{+} (\{0\} \times \St).$
\end{enumerate}

\end{lema}

\subsection*{Idempotent linear relations}

Idempotent linear relations were studied in \cite{ACMM1}, and multivalued projections were introduced by Cross and Wilcox in \cite{Cross} and later studied by Labrousse in \cite{Labrousse}.

\begin{Def} Let $T \in \lr(\X)$.  We say that $T$ is an \emph{idempotent} if $T^2=T.$ If, in addition,  $\ran T \subseteq \dom T$ we say that $T$ is a \emph{multivalued projection}. 
\end{Def}

\begin{prop}[\cite{Cross}, \cite{Labrousse}]
	Given $\M,\N$ subspaces of $\X$,
	$$\PMN:=\{(m+n,m):m\in\M, n\in \N\},$$
	is the unique multivalued projection with range $\M$ and kernel $\N.$ 
\end{prop}

We denote by $\Mp(\X)$ (resp. $\Id(\X)$) the set of multivalued projections (resp. idempotent relations) on $\X$. When there is no confusion as to the underlying space, we simply write $\Mp$ (and $\Id$) instead.

From now on, $\PMN$ is the multivalued projection with range $\M$ and kernel $\N.$ It is easy to check that $\dom \PMN=\M+\N$ and $\mul \PMN=\M\cap \N.$ When $\PMN$ is a projection, i.e., $\mul \PMN=\M\cap \N=\{0\}$, we write $P_{\M \pl \N}.$ We denote by $\cQ$ the set of projections.

\begin{prop}\label{repidemp}\cite[Corollaries 4.3, 4.4]{ACMM1} If $Q\in \Id(\X) $ then 
	$$Q=P_{\ran Q\cap\dom Q, \ker Q}\hat{+} (\{0\}\times\mul Q)$$
	
\end{prop}

Throughout the article, we will use that if $Q\in \Id(\X) $ then $Q^{-1}\in \Id(\X)$, see \cite[Propositions 2.2 and 2.4]{Labrousse}.

\section{Some properties of idempotent linear relations.}

In this section, we collect some properties of products of idempotent relations that will be useful in the sequel. We begin by establishing fundamental properties of the domains and kernels of the two factor products when one of them is an idempotent.

\begin{prop} \label{productproj} Let $T \in\lr(\X)$ and $ Q \in \Id(\X).$ If  $\dom T\cap \ran Q \subseteq \dom Q\cap\ran Q$ then 
\begin{align*}
\dom TQ&= \ker Q +\ran Q \cap \dom T \mbox{ and }
 \ker TQ=\ker Q + \ran Q \cap \ker T 
 \end{align*}
\end{prop}
\begin{proof} We prove  the first equality only since the other one can be proved in a similar fashion. 
	
	Let $x \in \dom TQ$ then $(x,z) \in Q$ and $(z,y) \in T$ for some $z, y \in \X.$ By Proposition \ref{repidemp}, we have that $x-z \in \ker Q$, since $z \in \dom T \cap \ran Q \subseteq \dom Q \cap \ran Q$. Hence $x=x-z+z \in \ker Q + \ran Q \cap \dom T.$ Conversely, let $x=x_1+x_2$ with $x_1 \in \ker Q$ and $x_2 \in \ran Q \cap \dom T.$  Then $(x_1,0) \in Q,$ and $(x_2,z) \in T$ and $(z',x_2)\in Q$ for some $z, z' \in \X.$ Since $x_2 \in \dom T \cap \ran Q\subseteq \dom Q \cap \ran Q,$ it holds that $(x_2-z',0) \in Q$. Therefore $(x,x_2)=(z',x_2)+(x_2-z',0)+(x_1,0) \in Q.$ So $(x,z) \in TQ$ and $x \in \dom TQ.$ 
\end{proof}

\begin{obs}
If $Q \in \Mp(\mc X)$, then $\ran Q \subseteq \dom Q$, and therefore the results of Proposition \ref{productproj} hold for any $T \in \lr(\mc X)$.
\end{obs}

By duality, we obtain the corresponding results for the range and multivalued part of the products of the form $QT$ with $Q$ an idempotent. 

\begin{cor}  Let $T\in\lr(\X)$ and $ Q \in \Id(\X).$ If $\ran T\cap \dom Q \subseteq \dom Q\cap \ran Q$ then 
\begin{align*}\ran QT&= \mul Q + \dom Q \cap \ran T \mbox{ and }\\
\mul QT&=\mul Q + \dom Q \cap \mul T.
 \end{align*}
\end{cor}

\begin{proof} Take inverses and use Proposition \ref{productproj}.
	\end{proof}

\begin{prop} \label{idemp} Let $P,Q \in \Id(\X).$ Then
$QP=Q$ if and only if $\ker P \subseteq \ker Q$ and $\dom Q \subseteq \dom P.$

\end{prop}
\begin{proof} Assume that $\ker P \subseteq \ker Q$ and $\dom Q \subseteq \dom P$. 
Let $(x,y) \in QP.$ Then there exists $z \in \X$ such that $(x,z) \in P$ and $(z,y) \in Q.$ Hence $z \in \ran P \cap \dom Q\subseteq \ran P \cap \dom P$ so that, by Proposition \ref{repidemp}, $(x-z,0) \in P.$ Then $x-z \in \ker P \subseteq \ker Q,$ and $(x,y)=(z,y)+(x-z,0) \in Q.$ This shows that $QP\subseteq Q. $ As for the reverse inclusion, let $(x,y) \in Q.$ Then $x \in \dom Q \subseteq \dom P$ so that $x=x_1+x_2$ with $x_1 \in \ran P \cap \dom P$ and $x_2 \in \ker P \subseteq \ker Q.$ Then, on one hand $(x_2,0) \in QP.$ On the other hand, applying again  Proposition \ref{repidemp}, $(x_1,x_1) \in P$ and $(x_1,y)=(x,y)-(x_2,0) \in Q$ so that $(x_1,y) \in QP.$ Hence $(x,y)=(x_1,y)+(x_2,0) \in QP.$ The converse is trivial.
\end{proof}

\begin{cor} \label{coridemp} Let $P,Q \in \Id(\X).$ Then $PQ=Q$ if and only if $\mul P \subseteq \mul Q$ and $\ran Q \subseteq \ran P.$
\end{cor}
\begin{proof}
Take inverses and apply Proposition \ref{idemp}.
\end{proof}

\begin{prop} \label{multproy} Let $P , Q \in \Mp(\X).$ If $P$ and $Q$ commute then
$$PQ=P_{\ran P \cap \ran Q, \ker Q + \ran Q \cap \ker P}.$$
\end{prop}

\begin{proof} 
	First, notice that  $\ran P \cap \ran T \subseteq \ran PT$ for all $T \in \lr(\X)$. In fact, if $y \in \ran T\cap \ran P$ then  $(x,y) \in T$ and $(x',y) \in P$ for some $x, x' \in \X.$ So that $(y,y) \in P$ because $P \in \Mp(\X).$ Then  $(x,y) \in PT$ and $y \in \ran PT.$
	
	Now, assume that  $P$ and $Q$ commute. Then $(PQ)^2=PQ$ so that $PQ\in \Id(\X). $ Moreover, $\ran P \cap \ran Q \subseteq \ran PQ =\ran QP\subseteq \ran P \cap \ran Q$, i.e., $\ran PQ =\ran P \cap \ran Q$ and, by Proposition \ref{productproj}, $\ker PQ=\ker Q + \ran Q \cap \ker P.$  Finally, since $\ran P \cap \ran Q \subseteq \dom P\cap \ran Q \subseteq \dom PQ,$ the result follows.
\end{proof}

	\begin{cor} If  $\PMN$ commute with some $Q\in\cQ$ with $\ran Q=\St$ then
		$$\St\cap(\M+\N)=\St\cap\M+\St\cap\N.$$
	\end{cor}
	\begin{proof} Let $Q=P_{\St \pl \T}$ with $\T=\ker Q.$
		By Proposition \ref{multproy}, $ \PMN P_{\St \pl \T}=P_{\M\cap\St, \T+\St\cap\N}$ and so $\dom (\PMN P_{\St \pl\T})=(\M\cap\St+\St\cap\N)\dot{+} \T$. On the other hand, by Proposition \ref{productproj}, $\dom (\PMN P_{\St \pl \T})=(\M+\N)\cap\St\dot{+}\T.$ Hence, $(\M\cap\St+\St\cap\N)\dot{+} \T=(\M+\N)\cap\St\dot{+}\T.$ Therefore, $\St\cap(\M+\N)=\St\cap\M+\St\cap\N.$
	\end{proof}

	\begin{lema}\label{lem1} Let $T\in \lr(\X)$ and $\St$ be a subspace of $\X.$ If $\ran T\cap \St\subseteq \mul T$ then $P_{\ran T,\St}T=T.$ Similarly, if $\dom T\subseteq \St+\ker T$ then $T=TP_{\St\cap\dom T,\ker T}.$
\end{lema}
\begin{proof}
	Assume that $\ran T\cap \St\subseteq \mul T,$ and let $(x,y)\in P_{\ran T,\St}T$. Then, $(x,w)\in T$ and $(w,y)\in P_{\ran T,\St}$ for some $w\in \ran T.$  Thus, $(w,w)\in P_{\ran T,\St}$ and so $w-y\in \mul P_{\ran T,\St}=\ran T \cap \St \subseteq \mul T $. Hence, $(0,w-y)\in T$ and so $(x,y)\in T.$ The opposite inclusion always holds. 
	
	Similarly, assume now that $\dom T\subseteq \St+\ker T$ and let $(x,y)\in T. $ Then $x=s+n$ with $s\in\St, n\in\ker T$. Hence, $(s,y)\in T$ and $(x,s)\in P_{\St\cap\dom T,\ker T}. $ Thus, $(x,y)\in TP_{\St\cap\dom T,\ker T}.$ Conversely, if $(x,y)\in TP_{\St\cap\dom T,\ker T}$ then $(x,s)\in P_{\St\cap\dom T,\ker T}$ and $(s,y)\in T$  where  $s\in \St\cap\dom T $ satisfies that $x=s+n$ for some $n\in\ker T$. Hence, $(n,0)\in T$ and $(x,y)=(s+n,y)\in T$.
\end{proof}

\section{Linear selections of idempotents relations}

In this section, we introduce the concept of linear selections, a notion that plays a central role in extending classical operator results to the more general framework of linear relations.

\begin{Def}[\cite{LN}] Let $T\in \lr(\X)$ . A linear subspace $T_0\subseteq T$ is called a {\it linear
	selection} or an {\it algebraic operator part} of $T$ if $T=T_0\hat{\dotplus}(\{0\}\times \mul T).$ 
\end{Def}

Clearly, linear selections are single-valued relations, i.e., operators. It is well known that every linear relation admits a linear selection \cite{LN}. For completeness, we include a proof of this fact. Recall that if $\N \subseteq \T$ are linear subspaces, there exists a subspace $\St$ such that $\T = \St \dotplus \N$. The existence of $\St$ follows from Zorn’s lemma. Consequently, given such $\N$ and $\T$, we can always define a projection $Q$ with $\dom Q = \T$ and $\ker Q = \N$.

\begin{prop}\label{seleccion}
	Let $T\in\lr(\X)$ and $Q$ be any projection with $\dom Q=\ran T$ and $\ker Q=\mul T.$ Then $T_0:=QT$ is a linear selection of $T$ with $\ran T_0=\ran Q$ and $\ker T_0=\ker T.$
\end{prop}
\begin{proof} Let $Q$ be any projection with $\dom Q=\ran T$ and $\ker Q=\mul T$, i.e., $Q=P_{\St//\mul T}$ for some $\St$ such that $\ran T=\St\dotplus\mul T.$ Define $T_0:=QT \in \op(\X)$. Note that, since $\dom Q=\ran T,$ $\dom T_0=\dom T$ and $\ran T_0=\ran Q.$ Moreover, if $x\in\ker T_0$ then $(x,y)\in T$ for some $y\in\ker Q=\mul T.$ Thus, $(0,y)\in T$ and so $(x,0)\in T.$ Then $\ker T_0\subseteq \ker T$ and so $\ker T_0=\ker T.$ 
	
  We claim that $T=T_0\hat{\dotplus}(\{0\}\times\mul T).$  In fact, if $(x,y)\in T_0=QT$ then $(x,z)\in T $ and $(z,y)\in Q$ for some $z\in\X.$ As $(y,y)\in Q$, then $z-y\in \ker Q=\mul T.$ Thus $(0,z-y)\in T$ and so $(x,y)\in T.$ Therefore $T_0\subseteq T$ and so $T_0\hat{\dotplus}(\{0\}\times\mul T)\subseteq T$. Now, as $\dom T=\dom T_0$ and $\mul T\subseteq \mul( T_0\hat{+}(\{0\}\times\mul T))$, we get that $T=T_0\hat{+}(\{0\}\times\mul T).$ Finally, if $(0,y)\in T_0$ with $y\in\mul T$ then $y\in \ran Q\cap \ker Q=\{0\}$. 
\end{proof}

From now on, we denote by  $\ls(T)$ the set of linear selections of $T\in \lr(X)$.
The following result characterizes the linear relations that have projections as linear selections. 

\begin{prop}\label{op-decomp} The following hold:

	\begin{equation}\label{super3}
		\{T\in\lr(\X): T\subseteq T^2\}=\{P_0 \hat{\dotplus}(\{0\}\times \T): P_0\in\cQ, \T\subseteq \X\}.
	\end{equation}
\begin{equation}\label{selecIdemp}
	\Id(\X)=\{P_0\hat{\dotplus}(\{0\}\times \T): P_0\in\cQ \; {\text and} \; \T\cap \dom P_0=\ker P_0\cap (\ran P_0+\T) \}.
\end{equation}
	\begin{equation}\label{selecMP}
		\Mp(\X)=\{P_0\hat{\dotplus}(\{0\}\times \T): P_0\in\cQ \; {\text and} \ \T\subseteq \ker P_0\}.
	\end{equation}

\end{prop}
\begin{proof}
	Let us begin by proving (\ref{selecMP}). Let $P\in \Mp(\X)$ and $Q$ be a projection with $\dom Q=\ran P$ and $\ker Q=\mul P.$ Define $P_0:=QP$. By Proposition \ref{seleccion}, $P=P_0\hat{\dotplus}(\{0\}\times\mul P).$ From this, $P_0\subseteq P$ and so $P_0^2=QP P_0\subseteq QP^2=QP=P_0.$ We claim that $\dom P_0\subseteq\dom P_0^2.$ In fact, if $x \in \dom P_0$ then $(x,y) \in P_0$ for some $y \in \ran P_0\subseteq \ran Q \subseteq \dom Q=\ran P \subseteq \dom P=\dom P_0$. Hence $(y,z) \in P_0$ for some $z \in \mc X$ and $(x,z) \in P_0^2,$ i.e., $x\in\dom P_0^2$ as desired. Therefore, $P_0=P_0^2.$  Finally, by Proposition \ref{productproj},  $\ker P_0=\ker P + \ran P \cap \ker Q=\ker P+ \dom Q \cap \ker Q=\ker P + \ker Q\supseteq \ker Q =\mul P.$ 
	
	For the opposite inclusion, let $T:=P_{\St\pl\N}\hat{\dotplus}(\{0\}\times \T)$ with $\St$ and $\N$ subspaces of $\mc X$ such that $\T\subseteq \N$. Thus, $\ran T=\St+\T\subseteq \St+\N=\dom T,$ and $\ker T=\N.$ 
	Then, by applying Lemma \ref{sumasubs}, we get that $T^2=TP_{\St\pl\N} + T(\{0\} \times \T)$, and so by Lemma \ref{lem1}, $T^2=T+ (\{0\} \times \mul T)=T.$
	Therefore $T\in\Mp(\X).$
	
As for \eqref{super3}, let $T:=P_{\St \pl \N}\hat{\dotplus}(\{0\}\times \T).$ Then, by \cite[Corollary 3.14]{ACMM1}, $T\subseteq T^2.$ This gives the inclusion $\supseteq$ in \eqref{super3}.
The opposite inclusion follows by \cite[Corollary 3.14]{ACMM1} and \eqref{selecMP}.
Finally,  \cite[Proposition 3.17]{ACMM1} and  (\ref{super3}) give \eqref{selecIdemp}.
\end{proof}

A linear relations  $T$ such that $T\subseteq T^2$ is called {\it super-idempotent}. This class was studied in \cite{ACMM1}.

\begin{cor} It holds that:
	$$\lr(\X, \Y)=\Mp \op(\X,\Y)$$
\end{cor}
\begin{proof}
	Let $T\in \lr(\X,\Y)$ and $Q$ be a projection with $\dom Q=\ran T$ and $\ker Q=\mul T.$ Then, by Proposition \ref{seleccion}, $T_0=QT\in\op(\X,\Y)$ and $T=T_0\hat{\dotplus}(\{0\}\times \mul T).$ Thus, $T=QT\hat{\dotplus}(\{0\}\times \mul T)=QT\hat{\dotplus}(\ker QT \times \mul T)=(Q\hat{\dotplus}(\{0\}\times \mul T))QT=(Q\hat{\dotplus}(\{0\}\times \mul T))T_0\in \Mp \op(X,\Y)$, since $Q\hat{\dotplus}(\{0\}\times \mul T)\in\Mp$ because of Proposition \ref{op-decomp}. 
\end{proof}

\section{Factorizations of linear relations}

R.G. Douglas showed in \cite{Douglas} that, for two Hilbert space bounded linear operators $R$ and $S,$ the existence of a bounded linear operator $T$ satisfying  $R = ST$ is equivalent to $\ran R\subseteq \ran S.$ This result  plays a crucial role in understanding different operator factorizations (\cite{AAC},  \cite{ACM}, \cite{Cor}, among others).

 Here we present the proof of Douglas' lemma in the context of operators between linear spaces.

\begin{lema}\label{Douglasoperadores}Let $R \in \op(\X, \Z)$ and $S\in \op(\Y, \Z).$ Then the following are equivalent:
	\begin{enumerate}
		\item There exists $T \in \op(\X, \Y)$ such that $R=ST;$
		\item $\ran R \subseteq \ran S$.
	\end{enumerate}
\end{lema}
\begin{proof}
	 Assume that $\ran R \subseteq \ran S$. Let $\N$ be such that $\dom S=\N \dotplus \ker S.$ Define $T:=\{(x,y)\in\dom R\times \N: \exists z\in \Z \; s.t. \; (x,z)\in R, (y,z)\in S \}$. It is straightforward that $T\in\lr(\X,\Y).$ Moreover, if $(0,y)\in T$ then $(y,0)\in S$ since  $R \in \op(\X, \Z)$. Thus $y\in \N\cap\ker S$, and so $y=0.$ Hence,  $T\in\op(\X,\Y).$
	 
	 Finally,  $R=ST.$ In fact, if $(x,z')\in ST$ then $(x,y)\in T$ and $(y,z')\in S$ for some $y\in \Y.$ Then, $(x,z)\in R$ and $(y,z)\in S$ for some $z\in\Z$.  But since $S\in \op(\Y, \Z),$ we have that  $z=z'$. Thus, $(x,z')\in R$ and so $ST\subseteq R.$ 
	 Conversely, if $(x,z)\in R$ then, as $\ran R \subseteq \ran S$, $(y,z)\in S$ for some $y\in \Y.$ Without loss of generality we can consider $y\in\N$. Thus, $(x,y)\in T$ and so $(x,z)\in TS,$ as desired.
	 
	 The converse is straightforward.  
\end{proof}

A natural question arises as to whether a Douglas-type result can be established for linear relations. In this context, several variants of Douglas’ factorization lemma may be considered. For instance, given $R \in \lr(\X, \Z)$ and $S \in \lr(\Y, \Z)$, one may ask whether there exists $T \in \lr(\X, \Y)$ (or $T \in \op(\X, \Y)$) such that $R \subseteq ST$ or $R = ST$. These cases, among others, were studied in \cite{PS} and \cite{Sando}.

The following result, due to Sandovici et al. \cite{Sando}, will be useful in what follows. We provide an alternative proof using linear selections.

\begin{thm}[{\cite[Theorem 1]{Sando}}] \label{Douglasop}
	Let $R\in\lr(\X,\Y)$ and $S\in \lr(\Z,\Y).$ The following are equivalent:
	\begin{enumerate}
		\item There exists  $T_0\in\op(\X,\Y)$ such that $R=ST_0.$
		\item $\ran R\subseteq \ran S$ and $\mul R=\mul S$.
	\end{enumerate}
\end{thm}

\begin{proof}
$i)\Rightarrow ii)$:	Assume that $R=ST_0$ with $T_0$ an operator. Then, $\ran R \subseteq \ran S$ and $\mul R = S(\mul T_0)=S(\{0\})=\mul S.$
	
$ii)\Rightarrow i)$: If $\ran R\subseteq \ran S$ and $\mul R=\mul S$ then there exist $\St,\T$ such that $\ran R=\St\dotplus\mul R$ and $\ran S=\T\dotplus\St\dotplus\mul S.$ Therefore, by Proposition \ref{seleccion}, there exist operators $R_0,S_0$ such that $R=R_0\hat{\dotplus}(\{0\}\times\mul R)$ with $\ran R_0=\St$, and $S=S_0\hat{\dotplus}(\{0\}\times\mul S)$
with $\ran S_0=\St\dotplus\T.$ As $\ran R_0\subseteq \ran S_0$,  by Lemma \ref{Douglasoperadores}, there exists an operator $T_0$  such that $R_0=S_0 T_0.$ Thus, $R=R_0\hat{\dotplus}(\{0\}\times\mul R)=S_0 T_0\hat{\dotplus}(\{0\}\times\mul R)=S_0 T_0\hat{\dotplus}(\{0\}\times\mul S)=S_0 T_0\hat{\dotplus}(\ker T_0 \times\mul S)=(S_0 \hat{\dotplus}(\{0\}\times\mul S))T_0=ST_0,$
and $i)$ holds.
 
 Moreover
 \begin{equation}\label{nucleoSR}
 	\{x: S(x)=R(x)\}=\ker(S_0-R_0)
 \end{equation}  In fact, $S(x)=R(x)$ if and only if $S_0(x)+\mul S=R_0(x)+\mul R.$ Thus, given $m\in\mul S$ there exits $n\in\mul R$ such that $S_0(x)+m=R_0(x)+n,$ so $S_0(x)-R_0(x)=n-m\in \ran S_0\cap\mul S=\{0\}$, i.e., $x\in \ker(S_0-R_0).$ 
 Equality (\ref{nucleoSR}) will be useful later.
\end{proof}

\begin{obs}\label{operador} Let $R\in\lr(\X,\Y)$ and $S\in \lr(\Z,\Y)$ be such that $\ran R\subseteq \ran S$ and $\mul R=\mul S$. An alternative method for constructing an operator solution of $R=SX$ is as follows: 

By Proposition \cite[Proposition 2]{Sando}, there exists $T\in \lr(\X,\Z)$ such that $R=ST$. Without loss of generality, we can assume that $\dom T=\dom R.$ 
Let $\St$ be an algebraic complement of $\mul T\cap \dom S$ in  $\ran T\cap \dom S,$  so that
$$\ran T\cap \dom S=\St\dotplus \mul T\cap \dom S.$$
Let $Q$ be the projection with  $\ran Q=\St$ and $\ker Q=\mul T\cap \dom S$. Define $T_0:=QT.$ We claim that $T_0$ is an operator such that $R=ST_0.$ First, let us see that $T_0$ is an operator. To this end, let $(0,y)\in T_0$ so that there exists $z\in\Z$ such that $(0,z)\in T$ and $(z,y)\in Q.$ Thus, $z\in \mul T\cap \dom Q=\mul T\cap \dom S=\ker Q$ and so $y=Qz=0.$ Next, let us prove that $R=ST_0$ or, equivalently, that $ST=ST_0=SQT.$ For this, first notice that $\ker Q \subseteq \ker S.$ Indeed, if $y\in \ker Q= \mul T\cap \dom S$ then $(0,y)\in T$ and $(y,w)\in S$ for some $w\in \Y.$ Thus, $(0,w)\in ST=R$ and, since $\mul R=\mul S$, we get that $(0,w)\in S$. Then $(y,0)=(y,w)-(0,w) \in S.$ Now, let $(x,y)\in ST$ so that $(x,w)\in T$ and $(w,y)\in S$ for some $w\in \ran T\cap\dom S=\dom Q.$ Thus, $w=Qw+(I-Q)w$ where $(I-Q)w\in\ker Q.$ Thus, $((I-Q)w, 0)\in S$ and so $(Qw,y)\in S.$ Summarizing, $(x,w)\in T$, $(w, Qw)\in Q$ and $(Qw,y)\in S.$ Therefore, $(x,y)\in SQT=ST_0.$ For the opposite inclusion, it suffices to note that $SQ\subseteq S$. In fact,  if $(x,y)\in SQ$ then $(x,Qx)\in Q$ and $(Qx,y)\in S.$ But, $((I-Q)x,0)\in S$, so $(x,y)\in S.$ The proof is complete.

\end{obs}

Our first goal in this section is to characterize those linear relations $R$ and $S$ for which  there exists $Q \in \cQ$ such that $R = SQ$. To this end, we begin by studying the operator case.

\begin{prop}\label{SolDQ}
Let $R,S\in \op(\X).$ 
There exists $Q\in\cQ$  such that $R=SQ$ if and only if $\dom R = \ker(S-R)+\ker R.$
\end{prop}
\begin{proof}
Let $Q$ be a projection such that $R=SQ$. Then, $SQ=RQ$ and so $\ran Q\subseteq \ker(S-R).$ From this, $\dom R\subseteq\dom Q=\ran Q\dotplus\ker Q\subseteq \ker(S-R)+\ker R.$ Conversely, if $\dom R = \ker(S-R)+\ker R$ then there exists a subspace $\St$ such that $\dom R=\St\dotplus\ker R$ and $\St\subseteq \ker(S-R).$ Let $Q:=P_{\St \pl \ker R}.$ Since $\ran Q \subseteq \dom S,$ then $R=RQ=(R-S+S)Q=(R-S)Q+SQ=SQ.$
\end{proof}

\begin{thm}\label{solQrl}
Let $R,S\in \lr(\X)$. The following are equivalent:
\begin{enumerate}
\item There exists $Q\in\cQ$  such that $R=SQ$.
\item  $\dom R = \{x: S(x)=R(x)\}+\ker R$ and $\mul R=\mul S.$
\end{enumerate}
\end{thm}
\begin{proof} 

Assume that there exists $Q\in\cQ$ such that $R=SQ.$ Then, $\ran R\subseteq \ran S$ and $\mul S=\mul R.$ Therefore, by the proof of Theorem \ref{Douglasop}, there exist operators $R_0,S_0$ such that $R=R_0\hat{\dotplus}(\{0\}\times\mul R)$  and $S=S_0\hat{\dotplus}(\{0\}\times\mul S)$
with $\ran R_0\subseteq \ran S_0$, and  $\{x: S(x)=R(x)\}=\ker(S_0-R_0)$. 

Consequently, $R_0\hat{\dotplus}(\{0\}\times\mul R)=(S_0\hat{\dotplus}(\{0\}\times\mul S))Q=S_0Q\hat{\dotplus}(\{0\}\times\mul S).$ Now, since $\mul R=\mul S$ and $\ran R_0\subseteq \ran S_0$, we get that $R_0 \subseteq S_0 Q.$ In fact, let $(x,y)\in R_0$ then $y=S_0Qx+s$ for some $s \in \mul S.$ Then $y-S_0Qx=s \in \ran S_0 \cap \mul S=\{0\}.$ So that $(x,y) \in S_0Q.$
 Hence $R_0=S_0Q.$ Thus, by Proposition \ref{SolDQ}, $\dom R=\dom R_0= \ker(R_0-S_0)+\ker R_0=\ker(R_0-S_0)+\ker R=\{x: S(x)=R(x)\}+\ker R.$ 

Conversely, assume that $ii)$ holds. We claim that $\ran R\subseteq \ran S.$ Indeed, if $y\in\ran R$ then $(x,y)\in R$ for some $x=x_1+x_2$ with $x_1\in \{x: S(x)=R(x)\}$ and $x_2\in \ker R. $ Thus, $(x_1,y)=(x-x_2,y)\in R$, i.e., $y\in R(x_1)=S(x_1)$ and so $y\in\ran S.$ By the proof of Theorem \ref{Douglasop}, there exist operators $R_0,S_0$ such that $R=R_0\hat{\dotplus}(\{0\}\times\mul R)$  and $S=S_0\hat{\dotplus}(\{0\}\times\mul S),$ with $\ran R_0\subseteq \ran S_0$ and $\{x: S(x)=R(x)\}=\ker(S_0-R_0)$. Thus,  $\dom R_0=\ker(S_0-R_0)+\ker R_0.$ So that, by Proposition \ref{SolDQ}, there exists a projection $Q$ such that $R_0=S_0Q.$ Therefore, as $\mul R=\mul S$, $R=R_0\hat{\dotplus}(\{0\}\times\mul R)=S_0Q\hat{\dotplus}(\{0\}\times\mul S)=(S_0\hat{\dotplus}(\{0\}\times\mul S))Q=SQ.$
\end{proof}

Our second goal is to study the reverse factorization $R = QS$, where $R$ and $S$ are linear relations and $Q$ is a projection. This problem requires a deeper analysis. We present an initial approach based on linear selections.

\begin{prop}\label{SolIQ}
	Let $R,S\in \op(\X).$ The following are equivalent:

	\begin{enumerate}
	\item There exists $Q\in\cQ$  such that $R=QS$. 
	\item $Q_1:=Q_{\ran R \pl \ran(S-R)}$ is well-defined and $R=Q_1S.$ 
	\item  $\ran(S-R)\cap \ran R=\{0\}$ and $\dom R=S^{-1}(\ran R\dotplus\ran(S-R))$.
	\end{enumerate}
\end{prop}
\begin{proof}
$i)\Leftrightarrow ii)$:	Let $Q\in\cQ$  such that $R=QS$. Then, $QR=QS$ and so $\ran(S-R) \cap \ran R\subseteq \ker Q \cap \ran Q=\{0\}.$ Thus, $Q_1:=Q_{\ran R \pl \ran(S-R)}$ is well-defined. Moreover, since $\dom Q_1=\ran R+\ran(S-R)\subseteq \ran Q+\ker Q=\dom Q,$ then $\dom Q_1S\subseteq\dom QS=\dom R.$ Let us see that $R \subseteq Q_1S.$ In fact, if $x \in \dom R \subseteq \dom S$ then $(R-S)x \in \ker Q_1$ so that $Rx=Q_1Rx=Q_1Sx.$ Now, since $\dom Q_1S\subseteq \dom R,$ we get indeed that $R=Q_1S.$ The converse is straightforward.

$ii)\Leftrightarrow iii)$: As $Q_1$ is well-defined then $\ran(S-R)\cap \ran R=\{0\}$. Furthermore, from $R=Q_1S$ we get that $\dom R=\dom(Q_1 S)=S^{-1}(\dom Q_1)=S^{-1}(\ran R\dotplus\ran(S-R)).$
Conversely, as $\ran(S-R)\cap \ran R=\{0\}$ then $Q_1:=Q_{\ran R \pl \ran(S-R)}$ is well-defined. We claim that $R=Q_1S.$ In fact, if $x \in \dom R \subseteq \dom S$ then $(R-S)x \in \ker Q_1$ so that $Rx=Q_1Rx=Q_1Sx.$ But, $\dom (Q_1S)=S^{-1}(\ran R+\ran(S-R))=\dom R.$ Therefore $R=Q_1S.$
\end{proof}

\begin{thm}\label{SolQ2}
	Let $R,S\in \lr(\X).$ The following are equivalent:

	\begin{enumerate}
	\item There exists $Q\in\cQ$  such that $R=QS$ and $Q(\mul S)=\mul R.$
	\item There exists $Q\in\cQ$  such that $R_0=QS_0$ for some $R_0\in\ls(R)$ and $S_0\in\ls(S)$, and $Q(\mul S)=\mul R.$
	\end{enumerate}
\end{thm}
\begin{proof}
$i)\Rightarrow ii)$: Let $S=S_0\hat{\dotplus}(\{0\}\times\mul S)$ for some $S_0\in\ls(S).$ Then, $R=QS=Q(S_0\hat{\dotplus}(\{0\}\times\mul S))=QS_0\hat{\dotplus}(\{0\}\times Q(\mul S))=QS_0\hat{\dotplus}(\{0\}\times \mul R).$ Thus, $QS_0\in \ls(R)$. That is, $R_0:=QS_0$ for some $R_0\in\ls(R)$.

$ii)\Rightarrow i)$: Assume that $ii)$ holds. Then, $R=R_0\hat{\dotplus}(\{0\}\times\mul R)=QS_0\hat{\dotplus}(\{0\}\times Q(\mul S))=Q(S_0\hat{\dotplus}(\{0\}\times\mul S))=QS.$
\end{proof}

\section{The product of multivalued projections}

Our aim in this section is to characterize the set $\Mp^2 := \{ PQ : P, Q \in \Mp(\X) \}$. 
 We refer to \cite{ACM}, \cite{B}, \cite{Cor}, \cite{Rad}, and the references therein, for results about the product of projections.

We begin with a technical result.

\begin{lema} \label{propmp2} Let $T \in \Mp^2.$ Then there exist $P, Q \in \Mp(\X)$ such that $T=PQ$ with $\ran P=\ran T,$ $\mul P=\mul T,$ $\dom Q=\dom T$ and $\ker Q=\ker T.$  
\end{lema}
\begin{proof}  Let $T=EF$ with $E, F \in \Mp(\X)$. Then $\ran T\subseteq \ran E$ and $\mul E\subseteq \mul T.$ On one hand, let $\St:=\ker E+\mul T$ and $P:=P_{\ran T, \St}$. By Proposition \ref{idemp}, $PE=P$ because $\dom P=\ran T + \St=\ran T + \ker E \subseteq \ran E + \ker E=\dom E$ and $\ker E\subseteq \St=\ker P.$ Also, by Lemma \ref{lem1}, $PT=T$ because $\ran  T \cap \St= \mul T + \ran T \cap \ker E \subseteq \mul T + \mul E=\mul T. $ So  $$PF=PEF=PT=T.$$
Hence $\mul P\subseteq \mul T.$ But $\mul P=\ran T\cap \St=\ran T\cap(\ker E+\mul T)=\ran T\cap \ker E+\mul T\supseteq \mul T.$ Thus, $\mul P= \mul T.$ On the other hand, if $T=PF$ then $\ker F\subseteq \ker T$ and $\dom T\subseteq \dom F.$ Let $\W:=\ran F\cap \dom T$ and $Q:=P_{\W,\ker T}.$ By Corollary \ref{coridemp}, $FQ=Q$ because $\mul F=\ker F \cap \ran F \subseteq \ker T \cap \ran F=\mul Q$ and $\ran Q=\W \subseteq \ran F.$ Also, by Lemma \ref{lem1}, $TQ=T$ because $\dom T \subseteq \dom F=\ran F + \ker F \subseteq \ran F+ \ker T.$ So that $$PQ=PFQ=TQ=T.$$ Hence $\dom Q=\W+\ker T=\ran F\cap \dom T+\ker T\subseteq \dom T.$ Thus, $\dom T=\dom Q.$
\end{proof} 

In our next result we prove that every element in $\Mp^2$ can be expressed as the product of a multivalued projection, with prescribed range and multivalued part, and a projection with prescribed domain and kernel.

\begin{thm}\label{Zorn}
	
	It holds that
	$$\Mp^2=\Mp \cQ.$$
	Moreover, given  $T \in \Mp^2$ there exist $P\in \Mp$  and  $Q_0\in\cQ$ with $\ran P=\ran T,$ $\mul P=\mul T,$ $\dom Q_0=\dom T$ and $\ker Q_0=\ker T$ such that $T=PQ_0.$  
\end{thm}
\begin{proof}
	Let $T=PQ$ with $P,Q\in\Mp$ as in Lemma \ref{propmp2}, and let $E$ be a projection with $\dom E=\ran Q\cap \dom P$ and $\ker E=\mul Q\cap\dom P.$ Then, by Remark \ref{operador}, $Q_0:=EQ$ is an operator such that $T=PQ_0$ with $\dom Q_0=\dom T=\dom Q$ and $\ker Q_0=\ker T$. It remains to show that $Q_0$ is a projection. Notice that $Q_0\subseteq Q$. In fact, if $(x,y)\in Q_0=EQ$ then there exists $z\in\dom E$ such that $(x,z)\in Q$ and $y=Ez$. Thus, $z=y+(I-E)z$ where $(I-E)z\in\ker E\subseteq \mul Q.$ Hence, $(0,(I-E)z)\in Q$ and so $(x,y)\in Q.$ Therefore,  $Q_0^2=(EQ)Q_0\subseteq EQ^2=EQ=Q_0.$ Finally, $\ran Q_0 \subseteq \ran E \subseteq \dom E \subseteq \ran Q \subseteq \dom Q=\dom Q_0$ then $\dom Q_0^2=\dom Q_0$ and $Q_0^2=Q_0.$
\end{proof}

 \begin{cor}[\textnormal{cf.~\cite[Proposition 2.7]{ACM}}]The following conditions are equivalent:
\begin{enumerate}
\item $T\in\Mp^2;$
\item $\dom T=\{x: T(x)=P_{\ran T, \St}(x)\}+\ker T,$ for some subspace $\St$ such that $\St\cap\ran T=\mul T.$
\end{enumerate}
\end{cor}
\begin{proof}
By Theorem \ref{Zorn}, $T\in\Mp^2$ if and only if $T=P_{\ran T, \St}Q_0$ for some $Q_0\in\cQ$ and $\St$ such that $\ran T\cap\St=\mul T.$ But, by Theorem \ref{solQrl}, the last condition holds if and only if $\dom T=\{x: T(x)=P_{\ran T, \St}(x)\}+\ker T,$ for some subspace $\St$ such that $\St\cap\ran T=\mul T.$
\end{proof}

The next corollary shows that every relation in $\Mp^2$ admits a linear selection  that is a product of projections.

\begin{cor}\label{QQ} It holds that
	\begin{equation}
		\Mp^2=\{P_0Q_0\hat{\dotplus}(\{0\}\times \St):P_0,Q_0\in\cQ, \St\subseteq \ker P_0\}\nonumber
	\end{equation}
\end{cor}
\begin{proof}
	Let $T=PQ_0$ as in Theorem \ref{Zorn} and $P=P_0\hat{\dotplus}(\{0\}\times \mul T)$ a decomposition of $P$ as in Proposition \ref{op-decomp}, i.e., $P_0\in\cQ$ and  $\mul T\subseteq \ker P_0$. Thus, $T=(P_0\hat{\dotplus}(\{0\}\times \mul T))Q_0=P_0 Q_0\hat{\dotplus}(\{0\}\times \mul T)Q_0=P_0 Q_0\hat{\dotplus}(\ker Q_0\times \mul T)=P_0Q_0\hat{\dotplus}(\{0\}\times \mul T)$ as desired.
	
	For the opposite inclusion, let $T=P_0Q_0\hat{\dotplus}(\{0\}\times \mul T)$ with $P_0,Q_0\in\cQ$ and $\mul T\subseteq \ker P_0$. By Lemma \ref{propmp2}, we can assume that $\ker Q_0=\ker P_0Q_0.$ Therefore, $T=P_0Q_0\hat{\dotplus}(\{0\}\times \mul T)=P_0Q_0\hat{\dotplus}(\ker Q_0\times \mul T)=(P_0\hat{\dotplus}(\{0\}\times \mul T))Q_0.$ Now, by Proposition \ref{op-decomp}, $P_0\hat{\dotplus}(\{0\}\times \mul T)\in\Mp$  and so $T\in\Mp^2.$
\end{proof}

We conclude this section by studying $\Mp^2$ when $\X$ is finite dimensional.
To this end, we recall a classical result on products of idempotent matrices due to Ballantine \cite{B}.

\begin{thm}[\cite{B}] \label{B} Let $\mathcal{F}$ be an arbitrary field, $n$ and $k$ be arbitrary positive integers, and $S$ be an $n\times n$ matrix over $\mathcal{F}$. Then $S$ is a product of $k$ idempotent matrices over $\mathcal{F}$ if and only if $\dim \ran (I-S)\leq k \dim \ker S.$
\end{thm}

\begin{prop}
	Let $\X$ be finite dimensional and $T\in \lr(\mc X).$ If $T\in \Mp^2$ then $\dim \ran (T-I)\leq 2 \dim \ker T+\dim \mul T.$
\end{prop}

\begin{proof}
	Let $T\in\Mp^2.$ Then, by Corollary \ref{QQ}, $T=P_0 Q_0\hat{\dotplus}(\{0\}\times \mul T)$ with $P_0, Q_0\in\cQ.$ Hence, $T-I=(P_0 Q_0-I)\hat{\dotplus}(\{0\}\times \mul T)$and $\ran (T-I)=\ran (P_0Q_0-I)\dotplus \mul T$. Therefore, by applying Theorem \ref{B} and from the fact that $\ker P_0Q_0=\ker T$, we get that 
	$$\begin{aligned}
	\dim \ran (T-I)&=\dim \ran (P_0Q_0-I)+\dim \mul T\leq 2 \dim \ker P_0Q_0+\dim\mul T\\
	&=2 \dim \ker T+\dim\mul T. 
	\end{aligned}$$
\end{proof}

\subsection*{Acknowledgments}
M.~Contino is
supported by CONICET PIP 11220200102127CO.

\end{document}